\documentclass[10 pt]{amsart}
\usepackage{amssymb,amsfonts}
\usepackage[all,arc]{xy}
\usepackage{enumerate}
\usepackage{mathrsfs}
\usepackage{stmaryrd}
\usepackage{amssymb}
\usepackage{bbm}
\usepackage{faktor}
\usepackage{xfrac}
\usepackage{booktabs} 
\usepackage{array} 
\usepackage{paralist} 
\usepackage{verbatim} 
\usepackage{subfig} 
\usepackage[mathscr]{euscript}
\usepackage{lipsum}       
\usepackage{changepage}
\usepackage{mathtools}
\usepackage {tikz-cd}
\usepackage{bbm}
\usepackage{amsmath}

\usepackage{graphicx}
\graphicspath{{./Graphics/}}
\usepackage{gensymb}
\usepackage{wrapfig}
\usepackage{adjustbox}
\usepackage{float}
\usepackage{amsthm}
\usepackage[calc,showdow,english]{datetime2}
\usepackage{fancyhdr}
\usepackage{xcolor}
\usepackage[margin=1 in]{geometry}
\tikzset{%
    symbol/.style={%
        draw=none,
        every to/.append style={%
            edge node={node [sloped, allow upside down, auto=false]{$#1$}}}
    }
}
\usepackage{tikz}
\usetikzlibrary{calc}

\newcommand{\imgoverlay}[2]{%
  \begin{tikzpicture}
    \node[anchor=south west, inner sep=0] (img) at (0,0)
      {\includegraphics[width=.9\linewidth]{#1}};
    \begin{scope}[x={(img.south east)}, y={(img.north west)}]
      #2
    \end{scope}
  \end{tikzpicture}%
}

\newtheorem{thm}{Theorem}[section]

\newtheorem{lem}[thm]{Lemma}

\theoremstyle{definition}
\newtheorem{defn}[thm]{Definition}

\newtheorem{exmp}[thm]{Example}

\newtheorem{clm*}{Claim}

\theoremstyle{remark}
\newtheorem{rem}[thm]{Remark}

\newtheorem{warn}[thm]{Warning}

\usepackage[%
    colorlinks=true,
    pdfborder={0 0 0},
    linkcolor=blue,
    linktocpage=true,
    citecolor=blue
]{hyperref}

\title{ 4-dimensional Skein modules, Handle attachments, and Tangles}
\author{Gage Martin, Mary Stelow, and Mira Wattal}

\begin{document}

\begin{abstract}
    
Skein lasagna modules are a recent tool developed for the study of 4-manifolds. We provide general formula for 1-, 2-, and 3-handle attachments for skein modules defined with any functorial link theory in $S^3 \times I$ generalizing existing formula of Chen, Manolescu-Neithalath, Manolescu-Walker-Wedrich, and Ren-Willis. These formula are derived from a complete description of the gluing homomorphism on skein modules. For this description, we introduce a variation of these skein modules in the presence of distinguished 3-manifolds in the boundary. A similar construction was recently introduced independently by Blackwell-Krushkal-Luo.
\end{abstract}

\maketitle

\section{Introduction}\label{sec_intro}

Motivated by constructions in the study of 3-manifolds, Morrison-Walker-Wedrich~\cite{morrison_invariants_2022} introduced a version of a skein module for a 4-manifold $X$ with a framed, oriented link $L$ in its boundary: the skein lasagna module of $(X; L)$. This invariant is typically denoted by $S^\star(X; L)$, where the star indicates a choice of an input functorial link homology theory $H^\star$ for links in $S^3$. When such a theory is fixed in our exposition, we will denote in place of the star. For example, we will denote the input of Khovanov-Rozansky $\mathfrak{gl}_2$-homology as $S^2(X; L)$. When such a theory is ambiguous, we assume that it satisfies a few additional conditions, which we have included in Section \ref{sec_notation}. 

Since their introduction, skein lasagna modules have had applications to the study of 4-manifolds~\cite{ren_khovanov_2025-1,sullivan_bar-natan_2025,sullivan_kirby_2024}, most notably in Ren-Willis's gauge theory-free proof of the existence of homeomorphic but not diffeomorphic 4-manifolds~\cite{ren_khovanov_2025}. Many of these applications rely on understanding the behavior of skein lasagna modules under various handle attachments. This behavior was first described in the case of 2-handles for Khovanov-Rozansky $\mathfrak{gl}_n$-homology by Manolescu-Neithalath~\cite{manolescu_skein_2022}. Later formulae have also been worked out for 2-handles with other choices of link theories, as well as for 1- and 3-handles for $\mathfrak{gl}_n$-homology again~\cite{chen_floer_2022,manolescu_skein_2023,ren_khovanov_2025}.

We extend and package these results into formulae for 1-,2-, and 3-handle attachments for a general choice of functorial link theory. Our results can be summed into the following theorems. A more careful treatment of these results appears in Examples~\ref{1-handle},~\ref{2-handle}, and~\ref{3-handle} respectively.

\begin{thm} 
If $X_1$ is obtained from $X$ by the addition of a 1-handle along $S^0 \times B^3$ and $T$ is the tangle obtained from $L$ by cutting along the cocore of the 1-handle, then $S^\star(X_1; L)$ is the 0-th Hochschild homology of a bimodule associated to the triple: $(X; S^0 \times B^3; T)$. The bimodule structure is deduced from the action by an algebra associated to $S^0 \times B^3$. 
\end{thm}

\begin{thm}
If $X_1$ is obtained from $X$ by the addition of a 2-handle along $S^1 \times D^2$, then $S^{\star}(X_1; L)$ can be obtained by ``cabling" a module associated to the triple: $(X; S^1 \times D^2; \varnothing)$. The cabling relations are deduced from an action by an algebra associated to $S^1 \times D^2$. 
\end{thm}

\begin{thm}
If $X_1$ is obtained from $X$ by the addition of a 3-handle along $S^2 \times D^1$, then $S^{\star}(X_1; L)$ is the quotient of a module associated to the triple: $(X; S^2 \times D^1; \varnothing)$. The relations are deduced from an action by an algebra associated to $S^2 \times D^1$. 
\end{thm}

Our handle attachment formulae all follow as specific cases of a general gluing formula, which is Theorem \ref{thm_gluing}. As above, we have summed its technical content below.  

\begin{thm}
Let the 4-manifold and link pair $(X; T_1\cup T_2)$ be the result of gluing 4-manifold, 3-manifold, and tangle triples $(X_1; Y; T_1)$ and $(X_2; Y; T_2)$ along a 3-manifold $Y$. Then, the skein lasagna module of $(X; T_1\cup T_2)$ can be obtained as a bimodule tensor product of a left module associated to $(X_1; Y; T_1)$ and a right module associated to $(X_2; Y; T_2)$.
\end{thm}

The gluing formula involves the technology of a skein lasagna module associated to a parameterized triple, which we will introduce in Section \ref{sec_skeintriple}. It is a variation of the existing skein module invariants associated to 4-manifolds, concocted precisely for its versatility in describing the cutting and pasting of 4-manifolds. A similar construction was introduced independently by Blackwell-Krushkal-Luo to study trisections~\cite{blackwell_cornered_2025}.

\section{Notation}\label{sec_notation}

We always work with skein lasagna modules $S^\star(X; L)$ defined using functorial link homology theories that satisfy these minimal conditions:
\begin{enumerate}
\item the theories are strictly monoidal with respect to disjoint unions and tensor products; and 
\item the theories can be extended to framed, oriented links in \emph{abstract} $S^3$ and link cobordisms in \emph{abstract} $S^3 \times I$. 
\end{enumerate}
These conditions are borrowed and subsequently abridged from~\cite{morrison_invariants_2022,ren_khovanov_2025}. We often but not always denote Khovanov-Rozansky $\mathfrak{gl}_n$-homology by $\mathrm{KhR}_n$, Khovanov homology by $\mathrm{Kh}$, and any ambiguous functorial link homology theory by $H^\star$. Finally, we implicitly fix the ground ring to be a field $k$. 

\section{The skein lasagna module of a triple}\label{sec_skeintriple}

To formally state our gluing formula, we require an extension of skein lasagna modules to 4-manifold, 3-manifold, and tangle triples. Our extension comes equipped with an action by an algebra associated to the 3-manifold and also specializes to Khovanov's tangle invariant in a precise sense. 

Explicitly, let $X$ be an oriented 4-manifold with boundary; $\phi \colon Y \to \partial X$ be an orientation-preserving embedding of a compact, oriented 3-manifold $Y$ that extends to some fixed framing; and $T$ be a framed, oriented tangle in $\partial{X} \setminus \phi(\mathrm{int}(Y))$, whose boundary meets $\phi(\partial{Y})$ in a compatibly framed, oriented collection $P$. We additionally assume that $T$ intersects each component of $\phi(Y)$ generically with trivial algebraic intersection. This data specifies a module associated to the parameterized triple: $(X; (Y, \phi); T)$.

\begin{defn}[The skein lasagna module of a triple]\label{defn_skeintriple}
Let $T_Y \subset Y$ be any embedded, framed tangle whose boundary is the compatibly framed, oriented collection $-\phi^{-1}(P)$. To $T_Y$, we associate a skein lasagna module $S^\star(X; T \cup_P \phi(T_Y))$. Skein lasagna modules such as these assemble into a $k$-module:
\[
S^\star(X; Y; T) \coloneqq \bigoplus_{T_Y} S^\star(X; T \cup_P \phi(T_Y)),
\]
which we define to be the skein lasagna module of $(X; Y; T)$. 
\end{defn}

\begin{warn}\label{warn_parametrization}
Though the 3-manifold $Y$ in Definition \ref{defn_skeintriple} is specified by an explicit embedding and framing into $\partial{X}$, we suppress the parameterizing map in our notation to prioritize being less verbose. We will continue to abbreviate our notation in this way; however, the reader should remain eagle-eyed of the fact that $Y$ is not abstract.
\end{warn}

We can promote $S^\star(X; Y; T)$ to a right-module over a $k$-algebra associated to $Y$ and $P$, which we also take to be the skein lasagna module of $(X; Y; T)$. 

\begin{defn}[The algebra associated to a 3-manifold]\label{defn_algebra}
To a pair of framed, oriented tangles $T_Y^0$ and $T_Y^1$ embedded in $Y$ with the same boundary conditions as in Definition \ref{defn_skeintriple}, we associate a skein lasagna module $S^\star(Y \times I; -T^0_Y \cup_{\phi^{-1}(P)} T^1_Y)$, where $T^0_Y \subset Y \times {0}$, $T^1_Y \subset Y \times {1}$, and $-T^0_Y \cup_{\phi^{-1}(P)} T^1_Y$ is the framed, oriented link indicated by the gluing \[-T^0_Y \cup_{\phi^{-1}(P) \times 0} (-\phi^{-1}(P) \times I )\cup_{\phi(P) \times 1} T^1_Y.\] Skein lasagna modules such as these assemble into a $k$-algebra:
\[
S^\star(Y; P) \coloneqq \bigoplus_{T^0_Y, \, T^1_Y} S^\star(Y  \times I; -T^0_Y \cup_{\phi^{-1}(P)} T^1_Y),
\]
whose multiplication is defined on a pair of fillings $S_1$ and $S_2$ in $S^*(Y; P)$ according to the following rule:
\[
S_1 \cdot S_2 \coloneqq 
\begin{cases}
S_1 \cup_{S_1 \cap Y \times \{1\}} S_2 = S_1 \cup_{S_2 \cap Y \times \{0\}} S_2 & \text{if $S_1 \cap Y \times \{1\} = -(S_2 \cap Y \times \{0\})$} \\
0 & \text{otherwise}.
\end{cases}
\]
The algebra comes equipped with local units, which are specified by finite sums of cylinders on tangles with the same embedding and boundary conditions as before.
\end{defn}

The $k$-algebra in Definition \ref{defn_algebra} acts on $S^\star(X; Y; T)$ on the right in the evident way, according to a similar rule as multiplication in $S^\star(Y; P)$. 

\begin{rem}
If instead we were given an orientation-preserving map $\phi: -Y \to \partial X$ extending to a possibly different fixed framing and a compatibly framed, oriented tangle $T$ with boundary $-P$, then we could collect the skein lasagna modules $S^\star(X; -\phi(T_Y) \cup_P T)$, where $T_Y$ ranges over all framed, oriented tangles in $Y$ with compatibly framed, oriented boundary $-\phi^{-1}(P)$. These modules would assemble into a \emph{left} module over $S^\star(Y; P)$. We will henceforth refer to the right and left counterparts of $S^\star(X; Y; T)$ using the same notation. Context will make clear which module structure is present.  
\end{rem}

\begin{exmp}[Test case]\label{exmp_testcase}
One of the simplest examples of a skein lasagna module of $(X; Y; T)$ occurs when $X = B^4$, $Y = B^3$, and $\star = 2$. In this case, both the module associated to $(B^4; B^3; T)$ and the algebra associated to $(B^3; P)$ can be rewritten as direct sums of $\mathfrak{gl}_2$-homologies:
\[
S^2(B^4; B^3; T) \cong \bigoplus_{T_Y} \mathrm{KhR}_2(T \cup_P T_Y), \quad \text{and} \quad
S^2(B^3; P) \cong \bigoplus_{T^0_Y, \, T^1_Y} \mathrm{KhR}_2(-T^0_Y \cup_P T^1_Y).
\]
Note that we ignore the parametrization of $B^3$ in $S^3$ as it is unambiguous.

The above equivalences are consequences of the evaluation isomorphism
\[
\mathrm{ev} \colon S^2(B^4; L) \to \mathrm{KhR}_2(L),
\]
which was originally defined by Morrison-Walker-Wedrich in \cite{morrison_invariants_2022}. They can be further enhanced into mappings that intertwine the $S^2(B^3; P)$-action\footnote{or multiplication, in the case of $S^2(B^3; P)$} on either side, after defining a multiplication on the right hand side of the second equivalence. 
    
 \begin{figure}
    \centering
\includegraphics[width=0.8\linewidth]{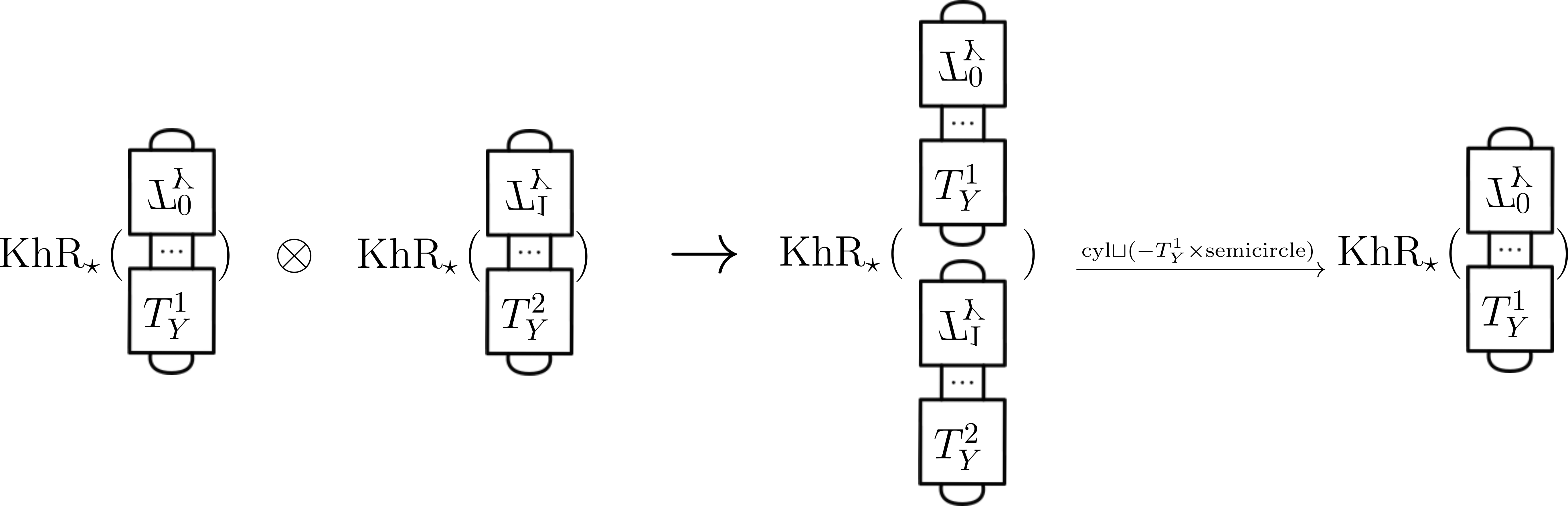}
    \caption{Multiplication in $\text{KhR}_\star$.}
    \label{fig:tangle_mult}
\end{figure}

Explicitly, let $v \in \mathrm{KhR}_2(-T^0_Y \cup_P T^1_Y)$ and $w \in \mathrm{KhR}_2(-T^2_Y \cup_P T^3_Y)$ be two elements in $S^\star(Y; P)$. If $T^1_Y \neq T^2_Y$, then set the multiplication $v \cdot w$ to be zero. Otherwise, define $v \cdot w$ to be the image of $v \otimes w$ under ``vertical composition", which is (again) borrowed from \cite{morrison_invariants_2022} and described by picture in Figure~\ref{fig:tangle_mult}. In words, vertical composition is induced by a link cobordism, which has ingoing boundary $-T^0_Y \cup_P T^1_Y \sqcup -T^1_Y \cup_P T^2_Y$ and outgoing boundary $-T^0_Y \cup T^2_Y$. It is cylindrical on the top and the bottom and $T^1_Y \times \text{semicircle}$ in the middle. 

Altogether, the above discussion can be summed by the below lemma, which connects the multiplication of fillings on the left with the vertical composition of $\mathfrak{gl}_2$-homologies on the right.

\begin{lem}\label{lem_evaluation}
The evaluation isomorphism specifies a $k$-algebra isomorphism:
\[
S^2(B^3; P) \cong \bigoplus_{T^0_Y, T^1_Y} \mathrm{KhR}_2(-T^0_Y \cup_P T^1_Y)
\]
and subsequently, an isomorphism of $S^2(B^3; P)$-modules:
\[
S^2(B^4; B^3; T) \cong \bigoplus_{T_Y} \mathrm{KhR}_2(T \cup_P T_Y).
\]
\end{lem}
\begin{proof}
Both isomorphisms in Lemma \ref{lem_evaluation} follow from similar pictures, so we only substantiate the first. The first isomorphism expresses a collection of commutative diagrams:
\[
\begin{tikzcd}[column sep=.6in]
S^2(B^3 \times I; -T^0_Y \cup_P T^1_Y) \otimes S^2(B^3 \times I; -T^1_Y \cup_P T^2_Y) \ar[r, "\mathrm{ev} \, \otimes \, \mathrm{ev}"] \ar[d, "- \cdot -", swap] & \mathrm{KhR}_2(-T^0_Y \cup_P T^1_Y) \otimes \mathrm{KhR}_2(-T^1_Y \cup_P T^2_Y) \ar[d, "- \cdot -"] \\
S^2(B^3 \times I; -T^0_Y \cup_P T^2_Y) \ar[r, "\mathrm{ev}"] & \mathrm{KhR}_2(-T^0_Y \cup_P T^2_Y),
\end{tikzcd}
\]
one for each triple of tangles $T^0_Y$, $T^1_Y$, and $T^2_Y$ embedded in $B^3$ satisfying the same boundary conditions as in Definition \ref{defn_algebra}. The fact that these diagrams commute follows from a series of topological moves that either represent gluings or equivalences between lasagna fillings.

Explicitly, let $F$ be a filling in $S^2(B^3 \times I, - T^0_Y \cup_P T^1_Y)$ and $G$ be a filling in $S^2(B^3 \times I, - T^1_Y \cup_P T^2_Y)$. Chasing $F \otimes G$ right first and down second through the diagram amounts to
\begin{enumerate}
\item respectively identifying $F$ and $G$ with the fillings $F'$ and $G'$ that are each specified by a single, large input ball, cylinder on the boundary link, and input label; and
\item mapping the tensor product of the labels corresponding to $F'$ and $G'$ to its image under the link cobordism map in Figure \ref{fig:tangle_mult}.
\end{enumerate}
On the other hand, chasing down first and right second means
\begin{enumerate}
\item replacing $F \otimes G$ with the union $F \cup G$, which is a filling in $S^2(B^3 \times I, - T^0_Y \cup_P T^2_Y)$;
\item identifying $F \cup G$ with a third filling specified by a single, large input ball, cylinder on the boundary link, and input label; and 
\item evaluating this third filling to its label in $\mathrm{KhR}_2(-T^0_Y \cup_P T^2_Y)$. 
\end{enumerate}
Notice that the Khovanov-Rozansky map induced by the intersection of $F \cup G$ with the large input ball in the second step of the second chase is the same as the composite of two Khovanov-Rozansky maps described by the two steps in the first chase. The first step expresses the tensor product of two maps: the first is induced by the intersection of $F$ with the input ball corresponding to $F'$ and the second is induced by the intersection of $G$ with the input ball corresponding to $G'$. The second step subsequently applies vertical composition to this tensor product. 

\begin{rem}
Both evaluation and vertical composition can be extended to general $H^\star$-homology with minimal elbow grease, as both of these operations are induced by link cobordisms in $S^3 \times I$. Subsequently, the above example can rewritten with $H^\star$ substituted for $\mathrm{KhR}_2$ everywhere. 
\end{rem}

\end{proof}

\end{exmp}

\begin{exmp}[The test case recovers Khovanov's tangle invariant]
We briefly elucidate how to relate Definition \ref{defn_skeintriple} to Khovanov's tangle invariant described in~\cite{khovanov_functor-valued_2002}. We largely omit a presentation of the relevant definitions and instead refer the reader to Khovanov's paper.

Recall that Khovanov assigns to each $n$-crossing diagram $D$ of an $(m,n)$-tangle $T$ a chain complex of $(H^m, H^n)$-bimodules. Denote the complex by $\mathcal{F}(D)$ and its homology by $H^\bullet(\mathcal{F}(D))$. The complex is an invariant of $T$ up to $(H^m, H^n)$-chain homotopy equivalence and it splits along $B^m \times B^n$. Likewise, the homology is an invariant of $T$ up to $(H^m, H^n)$-bimodule isomorphism and it splits as
\[
\bigoplus_{\mathclap{(a, b) \in B^m \times B^n}} \mathrm{Kh}({W(a)}D{b}).
\]
Here, $W(a)D{b}$ is the planar diagram obtained from $D$ by composing its top boundary (of $2m$ points) with a 180-degree reflection of $a$ and its bottom boundary (of $2n$ points) with $b$. 

The $(H^m, H^n)$-bimodule structure of $H^\bullet(\mathcal{F}(D))$ is defined similarly to the vertical composition depicted in Figure \ref{fig:tangle_mult} after remembering that $H^m$ and $H^n$ are direct sums of Khovanov homologies. In this case, the non-cylindrical parts of the cobordisms inducing the left action of $H^m$ and the right action of $H^n$ are cobordisms between flat tangles (instead of tangles in $B^3$). 

Specializing to Definition \ref{defn_skeintriple}, we instead let $T$ be a framed, oriented $(0,n)$-tangle with a compatibly framed, oriented boundary $P$. We additionally assume that $T$ satisfies the same embedding and boundary conditions that are true for the triple $(B^4; B^3; T)$ in Example \ref{exmp_testcase}. 

To each $b \in B^n$, we associate a skein lasagna module $S^2(B^4; T \cup_P T_b)$, where $T_b$ is a choice of an embedded, compatibly framed, oriented tangle in $B^3$ with boundary $-P$ that is represented by the diagram $b$. Skein lasagna modules such as these assemble into a $k$-(sub)module: 
\[
\bigoplus_{\mathclap{b \in H^n}} S^2(B^4;  T \cup_P T_b),
\]
which admits a right action by the $k$-(sub)algebra:
\[
\bigoplus_{\mathclap{(a,b) \in B^n \times B^n}} S^2(B^3 \times I; -T_a \cup_P T_{b}).
\]

With the above presentation, the following lemma feels imminent. 
\begin{lem}\label{lem_tangleinvariant}
Up to a sign ambiguity, there is an isomorphism of $k$-(sub)algebras:
\[
\bigoplus_{\mathclap{(a,b) \in B^n \times B^n}} S^2(B^3 \times I; -T_a \cup_P T_{b}) \cong H^n
\]
and subsequently, an isomorphism of right $H^n$-modules:
\[
\bigoplus_{\mathclap{b \in H^n}} S^2(B^4; T \cup_P T_b) \cong H^\bullet (\mathcal{F}(D)),
\]
\end{lem}
\begin{proof}
The proof is straightforward, following from a modified version of the proof of Lemma \ref{lem_evaluation}. Note that the two isomorphisms in Lemma \ref{lem_tangleinvariant} are specified by post-composing evaluation with the isomorphism connecting Khovanov-Rozansky $\mathfrak{gl}_2$-homology with Khovanov homology. This isomorphism is non-canonical with respect to link cobordisms, which explains the sign ambiguity. 
\end{proof}

\begin{rem}
Note that if $T$ was an $(m, 0)$-tangle with boundary $-P$ and we instead associated to each $a \in H^m$ a skein lasagna module $S^2(B^4; -T_a \cup_P T)$, then such skein lasagna modules would have assembled into a \emph{left} $H^m$-module, which is isomorphic as left $H^m$-modules to $H^\bullet(\mathcal{F}(D))$ (up to a sign ambiguity). 
\end{rem}
\end{exmp}

\section{The gluing homomorphism}\label{sec_gluinghom}

There is a gluing homomorphism that often appears in the minds of those topologists who study skein lasagna modules. This homomorphism was first (to the authors' knowledge) addressed in the specific case of a boundary connected sum by Manolescu-Neithalath in~\cite{manolescu_skein_2022} and later documented more generally by Ren-Wilis in~\cite{ren_khovanov_2025}. It relates the skein lasagna modules associated to two 4-manifolds, which share a common, smooth, codimension zero submanifold of their boundaries, with the skein lasagna module associated to their gluing along the distinguished submanifold. 

More precisely, let $(X_1; (Y, \phi_1); T_1)$ and $(X_2; (-Y, \phi_2); T_2)$ be two triples, each as in Definition \ref{defn_skeintriple}. We also assume that $\phi_1$ and $\phi_2$ extend to compatible framings and $\phi_1^{-1}(\partial{T_1}) =- \phi_2^{-1}(\partial{T_2})$ as compatibly framed, oriented collections of points in $Y$. Let $X = X_1 \cup_Y X_2$ denote the smooth gluing of the two 4-manifolds along a collar neighborhood diffeomorphic to $Y \times I$. The gluing formula as presented in \cite{ren_khovanov_2025} is a surjective mapping: 
\begin{equation}\label{eqn_RWgluing}
\bigoplus_{\mathclap{\substack{T_Y \subset Y \\ \partial{T_Y} = -\phi_1^{-1}(P)}}} S^\star(X_1; T_1 \cup_P \phi_1(T_Y)) \otimes S^\star(X_2; -\phi_2(T_Y) \cup_P T_2) \to S^\star(X; T_1 \cup_P T_2),
\end{equation}
where we simultaneously identify $P$ with $\partial{T_1}$ in $X_1$ and $-\partial{T_2}$ in $X_2$.\footnote{If we wanted to continue being pedantic, then the gluing $T_1 \cup_P T_2$ should likewise be interpreted as:
\[
T_1 \cup_{P \times {0}} P \times I \cup_{P \times {1}} T_2.
\]
}

In this section, we give a new presentation of the quotient computed by Equation \ref{eqn_RWgluing}, which may be self-evident to the skein lasagna topologist, particularly given the exposition in Section \ref{sec_skeintriple}. The target in Equation \ref{eqn_RWgluing} is isomorphic to a familiar bimodule tensor product:
\begin{thm}\label{thm_gluing}
\[
S^\star(X; T_1 \cup_P T_2) \cong S^\star(X_1; Y; T_1) \, \bigotimes_{\mathclap{S^\star(Y; P)}} \, S^\star(X_2; Y; T_2).
\]
\end{thm}

In the next subsection, we substantiate Theorem \ref{thm_gluing} with a topological proof. 

\subsection{Topological Picture}

There is an obvious map connecting the right-hand side of the isomorphism in Theorem \ref{thm_gluing} to the left-hand side. This map in shorthand is ``glue". In more words, we have a homomorphism:
\[
\Psi \colon S^\star(X_1; Y; T_1) \, \bigotimes_{\mathclap{S^\star(Y; P)}} \, S^\star(X_2; Y; T_2) \to S^\star(X; T_1 \cup_P T_2),
\]
which is described similarly to multiplication in $S^*(Y; P)$. If $\Sigma_1$ and $\Sigma_2$ are representatives of fillings respectively in $S^\star(X_1; Y; T_1)$ and $S^\star(X_2; Y; T_2)$, then the image of $\Sigma_1 \otimes \Sigma_2$ under $\Psi$ is defined as below:
\[
\Psi(\Sigma_1 \otimes \Sigma_2) \coloneqq 
\begin{cases}
\Sigma_1 \cup \Sigma_2 & \text{if } \partial \Sigma_1 \cap Y = - \partial \Sigma_2 \cap -Y  \\
0 & \text{otherwise.}
\end{cases}
\]
Note that the bottom condition in the above assignment is slightly redundant. If the boundaries of $\Sigma_1$ and $\Sigma_2$ were incompatible to begin with, then $\Sigma_1 \otimes \Sigma_2$ is already zero in the bimodule. 

To convince ourselves that $\Psi$ is an isomorphism, we construct an inverse whose operation in shorthand is ``cut". Explicitly, define 
\[
\Psi^{-1} \colon S^\star(X; T_1 \cup_P T_2) \to S^\star(X_1; Y; T_1) \, \bigotimes_{\mathclap{S^\star(Y; P)}} \, S^\star(X_2; Y; T_2)
\]
to be the homomorphism specified by the following assignment on fillings. To each filling $\Sigma$ in $S^*(X; T_1 \cup T_2)$, apply an isotopy to its underlying surface so that
\begin{enumerate}
\item\label{condition_1} the input balls of $\Sigma$ are supported away from a neighborhood of $Y$ and 
\item\label{condition_2} $\Sigma$ intersects $\phi_i(Y)$ transversely in some tangle $\phi_i(T_Y)$. 
\end{enumerate}
The above rearranging ensures that $\Sigma$ decomposes as 
\[
\Sigma = \Sigma_1 \cup_{\phi_i(T_Y)} \Sigma_2,
\]
where $\Sigma_1$ and $\Sigma_2$ are fillings respectively in $S^\star(X_1; Y; T_1)$ and $S^\star(X_2; Y; T_2)$. Evidently, we set $\Psi^{-1}(\Sigma)$ to be the tensor product of these fillings: $\Sigma_1 \otimes \Sigma_2$. 

At first glance, it appears that the definition of $\Psi^{-1}(\Sigma)$ depends on the choice of isotopy of $\Sigma$ arranging conditions \ref{condition_1} and \ref{condition_2}. The heart of proving Theorem \ref{thm_gluing} is rendering this choice unambiguous.

\begin{proof}[Proof of \ref{thm_gluing} (or that $\Psi^{-1}$ is well-defined)]
The fact that $\Psi^{-1}$ is independent of the choice of isotopy follows from an argument that closely resembles that for Theorem 1.4 in~\cite{manolescu_skein_2022}. Nevertheless, we recount the details with minor technical updates, which are relevant for our context. 

Let $\Sigma'$ be a different isotopic representative of $\Sigma$ satisfying conditions \ref{condition_1} and \ref{condition_2}, which decomposes as
\[\Sigma' = \Sigma'_1 \cup_{\phi_i(T'_Y)} \Sigma'_2.\]
There is a one parameter family of surfaces $\{\Sigma(t)\}_{t \in [0,1]}$ connecting $\Sigma = \Sigma(0)$ to $\Sigma' = \Sigma(1)$. For simplicity, we assume that this one parameter family either moves an input ball $B$ of $\Sigma$ through $\phi_i(Y)$ and keeps the rest of the skein fixed outside of this ball or fixes all of the input balls. (Any arbitrary one parameter family will be a sequence in both of these operations, up to isotopies that are supported away from a neighborhood of $\phi_i(Y)$.)

In the first case, we take $B$ to be a neighborhood of a point and assume that generically $B$ (or this point) moves from one side of $\phi_i(Y)$ to the other for a finite number of times. Let $t_j$ denote one such time and $[t_j - \epsilon, t_j + \epsilon]$ denote the sub-interval on which the point intersects $\phi_i(Y)$ under the isotopy just once at $t_j$. By construction, the intersections $\Sigma_i(t_j - \epsilon) \cap \phi_i(Y)$ and $\Sigma_i(t_j + \epsilon) \cap \phi_i(Y)$ are automatically transverse. 

If we take $\epsilon$ to be sufficiently small, then the intersection of $\Sigma(t_j - \epsilon)$ with $Y$ differs from the intersection of  $\Sigma(t_j + \epsilon)$ with $Y$ by replacing $\phi_i(Y)$ with a certain isotopic representative $\phi'_i(Y)$. Specifically, we require that the region between $\phi_i(Y)$ and $\phi'_i(Y)$ is diffeomorphic to $Y \times I$, and  $\Sigma(t_j - \epsilon)$ intersects $\phi'_i(Y)$ transversely in $\Sigma(t_j + \epsilon) \cap \phi_i(Y)$, and the intersection has boundary diffeomorphic to $-P$ (as framed, oriented collections). 

Altogether, we can realize $\Psi^{-1}(\Sigma(t_j - \epsilon))$ by cutting $\Sigma(t_j - \epsilon)$ along $\phi_i(Y)$ and $\Psi^{-1}(\Sigma(t_j + \epsilon))$ by cutting the same skein along $\phi'_i(Y)$. These cuts are further equivalent in the bimodule after intertwining the action of a filling in $S^*(Y; P)$ specified by cutting along both $\phi_i(Y)$ and $\phi'_i(Y)$. Hence, after proceeding inductively, we prove that $\Sigma_1 \otimes \Sigma_2$ and $\Sigma_1' \otimes \Sigma_2'$ are equivalent in the bimodule. 

In the second case, we consider the intersections of $\{\Sigma_t\}_{t \in [0, 1]}$ with $\phi_i(Y)$ at each time $t$: 
\[J_t = \Sigma_t \cap Y,\]
and collect these intersections into a surface with corners:
\[J = \bigcup_{\mathclap{t \in [0, 1]}} \, J_t \times \{t\} \subset \phi_i(Y) \times I.\]
By smoothing corners, we can assume that $J$ is a (vegetarian) filling in $S^*(Y; P)$.

From $J$, we build a second surface with corners by first thickening $\phi_i(Y)$ into a ``neck" or submanifold diffeomorphic to $Y \times I$ and then inserting $J$ into this neck:
\[\Sigma_J = \Sigma(0) |_{X_1} \cup J \cup \Sigma(1)|_{X_2}.\]
This surface can be similarly smoothed into a filling in $S^*(X; T_1 \cup_P T_2)$.

We glean two key insights from this procedure. The first is that $\Sigma(0)$ is isotopic to $\Sigma_J$ through an isotopy supported in $X_2$ and the second is that $\Sigma(1)$ is also isotopic to $\Sigma_J$ through an isotopy supported in $X_1$. The isotopy connecting $\Sigma(0)$ to $\Sigma_J$ at a given time $t$ is deduced (as before) by smoothing a surface with corners:
\[
\Sigma(0) |_{X_1} \cup \bigcup_{\mathclap{s \in [0, t]}} \, J_s \times \{s\}  \cup \Sigma(t)|_{X_2}.
\]
A similar construction describes the isotopy between $\Sigma(1)$ and $\Sigma_J$. In sum, these isotopies translate to bimodule equivalences relating $\Sigma_1 \otimes \Sigma_2$ to $\Sigma_1 \otimes J \cdot \Sigma_2'$ and $\Sigma_1' \otimes \Sigma_2'$ to $\Sigma_1 \cdot J \otimes \Sigma_2'$. The biaction further ensures a third crucial equivalence between $\Sigma_1 \otimes J \cdot \Sigma_2'$ and $\Sigma_1 \cdot J \otimes \Sigma_2'$---ultimately settling the desired claim. 
\end{proof}

\begin{rem}
It may be instructive for the skein lasagna topologist to convince themselves that either side of $\Psi$ is a worthy representative for the pushout of a particular diagram of skein lasagna modules. In Figure \ref{fig:Kernel}, we have depicted a toy picture of this diagram. 
\end{rem}

\begin{figure}[hbt!]
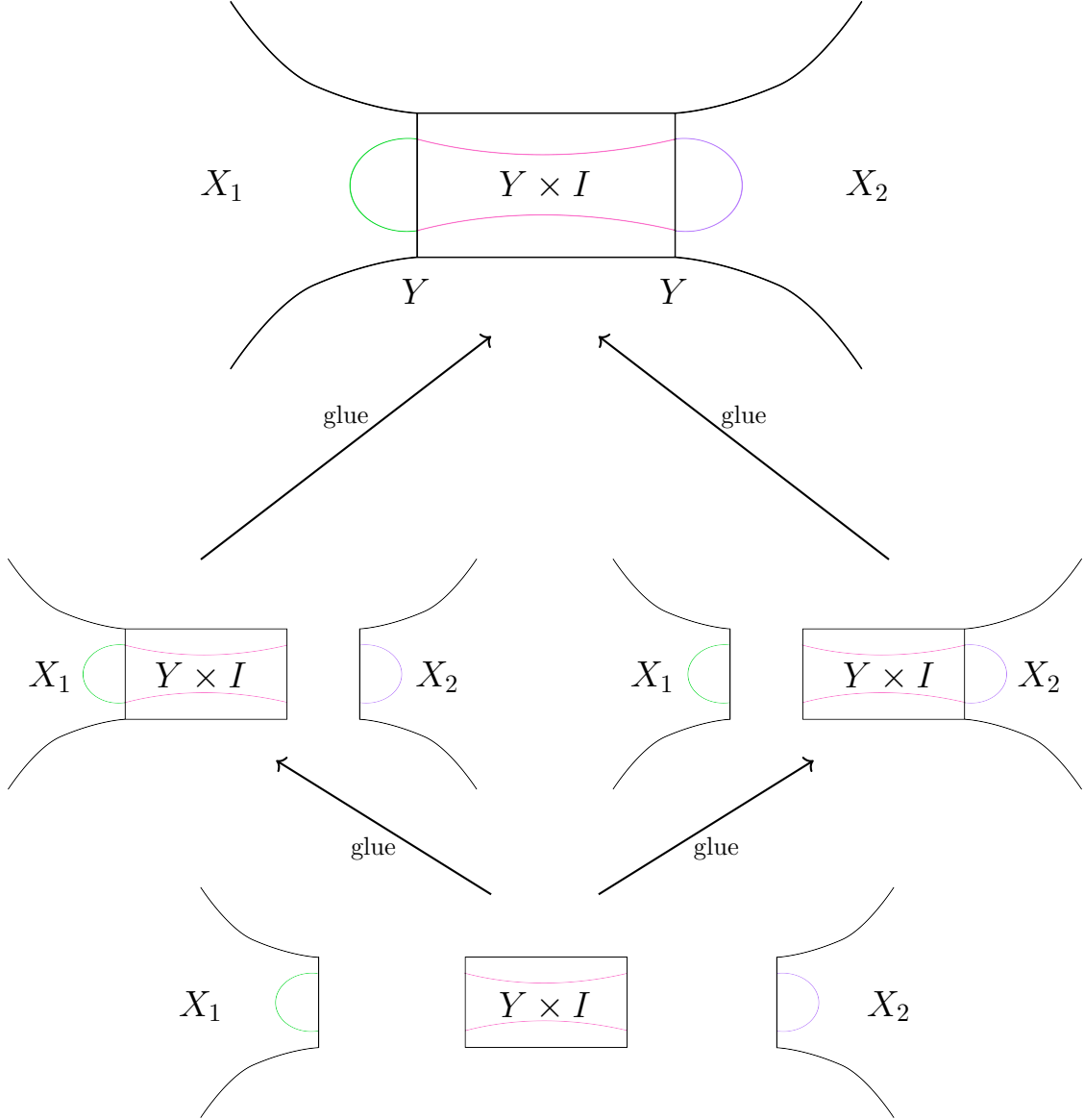

  \imgoverlay{bimodule_diagram.png}{
  \node[fill=none, inner sep=2pt] at (0.2,0.835) {{\LARGE $X_1$}}; 
  \node[fill=none, inner sep=2pt] at (0.5,0.835) {{\LARGE $Y \times I$}}; 
\node[fill=none, inner sep=2pt] at (0.8,0.835) {{\LARGE $X_2$}}; 
\node[fill=none, inner sep=2pt] at (0.62,0.74) {{\LARGE $Y$}}; 
\node[fill=none, inner sep=2pt] at (0.38,0.74) {{\LARGE $Y$}}; 
\node[fill=none, inner sep=2pt] at (0.18,0.395) {{\LARGE $Y \times I$}}; 
\node[fill=none, inner sep=2pt] at (0.82,0.395) {{\LARGE $Y \times I$}}; 
\node[fill=none, inner sep=2pt] at (0.04,0.395) {{\LARGE $X_1$}}; 
\node[fill=none, inner sep=2pt] at (0.4,0.395) {{\LARGE $X_2$}}; 
\node[fill=none, inner sep=2pt] at (0.6,0.395) {{\LARGE $X_1$}}; 
\node[fill=none, inner sep=2pt] at (0.96,0.395) {{\LARGE $X_2$}}; 

\node[fill=none, inner sep=2pt] at (0.18,0.1) {{\LARGE $X_1$}}; 
\node[fill=none, inner sep=2pt] at (0.82,0.1) {{\LARGE $X_2$}}; 
 \node[fill=none, inner sep=2pt] at (0.5,0.1) {{\LARGE $Y \times I$}}; 

\draw[->, thick]
    (0.18,0.5) -- (0.45,0.7)
    node[midway, above, yshift = 4pt] {$\text{glue}$};
 \draw[->, thick]
    (0.82,0.5) -- (0.55,0.7)
    node[midway, above,yshift=4pt] {$\text{glue}$};
\draw[->, thick]
    (0.45,0.2) -- (0.25,0.32)
    node[midway, below, xshift = -4pt] {$\text{glue}$};
 \draw[->, thick]
    (0.55,0.2) -- (0.75,0.32)
    node[midway, below,xshift=4pt] {$\text{glue}$};
}

  \caption{A pushout diagram of gluing homomorphisms.}
  \label{fig:Kernel}
\end{figure}

\section{Handle attachment formulae}\label{sec_handleformulae}

In this section, we apply Theorem~\ref{thm_gluing} to derive various handle attachment formulae for skein modules defined with a functorial link homology theory $H^{\star}$ of (potentially decorated) link cobordisms in $S^3 \times I$. At times, we may omit the boundary data from our notation for readability.

In their work on non-diffeomorphic 4-manifolds, Ren-Willis~\cite{ren_khovanov_2025} showed that the Lee skein theory satisfies a very similar 2-handle formula to the Manolescu-Neithalath analogue for the Khovanov skein theory~\cite{manolescu_skein_2022} and Chen showed that link Floer homology also satisfies a similar formula for 2-handle attachments~\cite{chen_floer_2022}. The general formula we provide for 2-handle attachments provides some additional clarity as to why these three previously known 2-handle formula share similarities. 

For completeness we remark briefly on 4-handles even though a topological general position argument easily shows that the addition of a 4-handle does not change a skein module~\cite{manolescu_skein_2022}. Let $X$ be the result of attaching a 4-handle to $X_1$ and take $Y$ to be the attaching region $S^3$. Here also $X_2 = B^4$ and both the skein modules $S^{\star}(X_1)$ and $S^{\star}(X_2)$ become modules over $S^{\star}(S^3\times I)$. For co-dimension reasons we know that both $L$ and any filling don't intersect the 4-handle so the relevant skein modules are $S^{\star}(X_1, L)$ and $S^{\star}(B^4, \emptyset) \cong k$. Additionally we know that the action of $S^{\star}(S^3\times I,\emptyset)$ on $S^{\star}(X_1, L)$ is exactly the $k$-module structure on $S^{\star}(X_1, L)$. Putting this together then the identification of $S^{\star}(X,L)$ with the quotient of $S^{\star}(X_1,L)\otimes S^{\star}(B^4, \emptyset)$ by the identification of the module actions recovers the previously known isomorphism $S^{\star}(X,L)\cong S^{\star}(X_1,L)$.

\begin{exmp}[1-handle attachment]\label{1-handle}
    
Let $X$ be the result of attaching a 1-handle to $X_1$ and take $Y$ to be the attaching region $S^0\times B^3 = -B^3 \sqcup B^3$. So $X_2 = B^1 \times B^3 \cong B^4$. We will compute $S^\star(X; L)$. 

By a general position argument, we can assume that $L$ intersects $X_1$ in a compatibly framed, oriented tangle $T$ and $X_2$ in $P \times I$, which is interpreted as a compatibly framed, oriented collection of core-parallel arcs. Note that both $S^*(X_1; -B^3 \sqcup B^3; T)$ and $S^*(X_2; -B^3 \sqcup B^3; P \times I)$ become modules over $S^\star(-B^3 \sqcup B^3; P)$, which is isomorphic as a $k$-algebra to $S^*(B^3; P)^\text{op} \times S^*(B^3; P)$.

Alternatively, this can be thought of as each having two distinct module structures over $S^\star(B^3)$: one for each of the attaching 3-ball regions. So, $S^\star(X_1)$ becomes both a right and left $S^\star(B^3)$-module and similarly for $S^\star(X_2)$. The left and right actions commute with one another other on both $S^\star(X_1)$ and $S^\star(X_2)$.

Theorem~\ref{thm_gluing} tells us that $S^\star(X; L)$ is obtained from $S^*(X_1) \otimes_{k} S^*(X_2)$ by setting the left action on $S^*(X_1)$ equal to the right action on $S^*(X_2)$ and the right action on $S^*(X_2)$ equal to the left action on $S^*(X_1)$. We can simplify this picture by recalling that we have an identification between $S^*(X_2)$ and $S^*(B^4)$.

Recalling this identification, given a filling $\Sigma_1 \otimes \Sigma_2$ in $S^\star(X,L)$, we can instead represent it as  $S_1 \cdot \Sigma_1 \cdot S_2$, for any choice fillings $S_1,S_2$ with $S_2 \cdot S_1 = \Sigma_2$. This insight tells us that in fact we have a surjective map $$f: \oplus_{T_i} S^{*} \left( X_1, T \cup (T_i \sqcup - T_i) \right) \to S^\star(X, L),$$
and the kernel is generated by the identification of the various fillings $S_1 \cdot \Sigma_1 \cdot S_2$ for different factorizations of $S_2 \cdot S_1 =\Sigma_2$. If we restrict to Khovanov-Rozansky $\mathfrak{gl}_n$-homology this exactly recovers 1-handle formula of Manolescu-Walker-Wedrich~\cite[Theorem~4.7]{manolescu_skein_2023}.
\end{exmp} 

\begin{exmp}[2-handle attachment]\label{2-handle}
Let $X$ be the result of attaching a 2-handle to $X_1$ and take $Y$ to be the attaching region $S^1\times B^2$. Here also $X_2 = B^2 \times B^2 \cong B^4$ and both the skein modules $S^\star(X_1)$ and $S^{\star}(X_2)$ become modules over $S^{\star}(S^1\times B^2)$. Using general position arguments, will may assume that $L$ is disjoint from the attaching region $Y$. Our tensor factors are initially
\[
    \mathcal{X}_1 = S^\star(X_1, Y, L); \quad \mathcal{X}_2 = S^\star(X_2, \emptyset); \quad \mathcal{Y} = S^\star(Y, \emptyset) .
\]
However we will first work to understand a smaller tensor product, which admits an obvious map into $\mathcal{X}_1 \otimes_{\mathcal{Y}} \mathcal{X}_2$. We first describe a subalgebra $\mathcal{Y}' $ of $\mathcal{Y}$: We will consider fillings of $S^1 \times B^2 \times I$ generated by fillings $S^1 \times T$ where $T$ where $T$ is either a crossingless $(n,n+2)$ tangle, with $n$ sheets between the two sides and a single semicircle between the two ends on the $n+2$ side or its reverse, an $(n+2, n)$-tangle. Additionally, we allow decorations on the sheet $S^1 \times T$ as determined by the functorial link theory $H^\star$. Finally, we allow fillings $S^1 \times B$ for braids $B$.  Call this algebra $\mathcal{Y}'$. There are right- and left- $\mathcal{Y}'$ submodules of $\mathcal{X}_1$ and $\mathcal{X}_2$;
\[
 \mathcal{X}_1' =   \bigoplus S^\star( X_1, L \cup K_{n_+, n}); \quad \mathcal{X}_2' = \bigoplus S^\star(X_2, U_{n_+, n_-}) .
\]
Note there is a map $\mathcal{X}_1' \otimes_{\mathcal{Y}'} \mathcal{X}_2' \to \mathcal{X}_1 \otimes_{\mathcal{Y}} \mathcal{X}_2 \cong S^\star(X,L)$.

Now there is a also map $\mathcal{X}_1' \to \mathcal{X}_1' \otimes_{Y'} \mathcal{X}_2'$ which simply tensors lasagna fillings $v$ in $X_1$ with correctly oriented cocore-transverse disks $d(v)$ in $X_2$. Take $v$ in $\mathcal{X}_1'$, and let $a \in \mathcal{Y}'$ be a (possibly decorated) $(n, n+2)$-sheet. Then in our bimodule tensor product we identify a simple tensor $v . a \otimes d(v.a)$ with $v \otimes a . d(v)$.  But $a. d(v)$ is the skein element $v \sqcup S$ where  $S$ is a possibly decorated sphere; the specifics of $H^{\star}$ determine an evaluation of these potentially decorated) spheres. Similarly if $a$ is a braided sheet then the action of $a$ on $d(v)$ is trivial. These two facts show that our map $\mathcal{X}_1' \to \mathcal{X}_1' \otimes_{Y}' \mathcal{X}_2'$ descends to a map on $\mathcal{X}_1' / \sim$, where $\sim$ is the quotient by the braiding relations along with the relations induced by consideration of (potentially decorated) spheres.

In total we now have a map 
\begin{align*}
\frac{\oplus S^\star(X_1, L \cup K_{n_+, n_-})}{\sim} \to S^\star(X,L).
\end{align*}
The proof that this map is an isomorphism follows the same steps as the similar argument in~\cite{manolescu_skein_2022} so details are omitted. Showing that the map is surjective can be done via general position arguments on fillings of $(X,L)$ with respect to the co-core of the 2-handle to describe a filling of $(X_1, L \cup K_{n_+, n_-})$ after cutting along the 2-handle. The argument for injectivity follows from considering how two different choices of an isotopy of a filling to realize transverseness to the co-core might differ in a neighborhood of the co-core. 

In the case that we restrict to Khovanov-Rozansky $\mathfrak{gl}_n$-homology we have exactly recovered the 2-handle formula of Manolescu-Neithalath~\cite{manolescu_skein_2022}. Restricting to link Floer homology and Lee homology recover the 2-handle formula of Chen and Ren-Willis respectively~\cite{chen_floer_2022,ren_khovanov_2025}.
\end{exmp}

\begin{exmp}[3-handle attachment]\label{3-handle}
Finally, we consider the case of 3-handle attachments. Let $X$ be the result of attaching a 3-handle to $X_1$ and take $Y$ to be the attaching region $S^2\times B^1$. Here also $X_2 = B^3 \times B^1 \cong B^4$. Define skein lasagna submodules and sub-algebra
\[\mathcal{X}_1 = S^{\star}(X_1, L) ; \quad \mathcal{X}_2 = S^{\star}(X_2, \emptyset) ; \quad \mathcal{Y}  = S^{\star} (S^2 \times B^2, \emptyset).
\]
A standard general position argument allows us to refine Theorem~\ref{thm_gluing} in this instance to say
\[
S^{\star}(X) \cong \mathcal{X}_1 \otimes_{\mathcal{Y}} \mathcal{X}_2 .
\]

Since $\mathcal{X}_2 \cong k$ we we learn that 
 $$S^{\star}(X) \cong \frac{\mathcal{X}_1}{\sim} $$ where the relation $\sim$ is given from the right action of $S^{\star}(S^2 \times B^2)$ on $\mathcal{X}_1'$ being set equal to the left action on $S^{\star}(B^4) = k$. In the case that we restrict to Khovanov-Rozansky $\mathfrak{gl}_n$-homology and use a computation of $S^n(S^2 \times B^2)$ we have exactly recovered the 3-handle formula of Manolescu-Walker-Wedrich~\cite[Remark~3.9]{manolescu_skein_2023}.
\end{exmp}

\bibliographystyle{plain}
\bibliography{references}

\end{document}